\documentclass[12pt]{article}

\usepackage{mathtools} 

\usepackage{float}
\usepackage{amsmath,amsthm}

\usepackage{amssymb}

\usepackage{mathrsfs}
\usepackage[mathcal]{eucal}

\newtheorem{theorem}{Theorem}[section]

\newtheorem{corollary}[theorem]{Corollary}
\newtheorem{definition}[theorem]{Definition}

\newtheorem{lemma}[theorem]{Lemma}
\newtheorem{proposition}[theorem]{Proposition}

\newcommand\R{\mathbb{R}}

\newcommand\abs[1]{|#1|}

\newcommand{\inner}[2]{\langle#1,#2\rangle}

\newcommand\norm[1]{\|#1\|}
\newcommand\Norm[1]{\left\|#1\right\|}

\newcommand\set[1]{\{#1\}}
\newcommand\Set[1]{\left\{#1\right\}}

\DeclareMathOperator*{\argmin}{argmin}

\title{H\"older continuity of the steepest descent direction
  for multiobjective optimization}

\author{B. F. Svaiter\thanks{IMPA, Estrada Dona Castorina 110,
    22460-320 Rio de Janeiro, Brazil ({\texttt{benar@impa.br}}) 
    tel: 55 (21) 25295112, fax: 55 (21)25124115.  }\hspace{.5em}
  \thanks{Partially supported by CNPq
    grant 306247/2015-1
    and by 
    FAPERJ grant Cientistas de Nosso Estado
    E-26/201.584/2014
  }}

\begin{document}
\maketitle

\begin{abstract}
  The aim of this manuscript is to characterize the continuity properties of
  the multiobjective steepest descent direction for smooth objective
  functions. 
  We will show that this direction is H\"older continuous with optimal exponent
  1/2.
  In particular, this direction fails to be Lipschitz continuous even for 
  polynomial objectives.

  \bigskip

  \bigskip


%



  \bigskip


  \bigskip

\end{abstract}

\section{Introduction}

In multiobjective optimization (MO) problems, many functions on the
same argument have to be simultaneously minimized. Since in general
there is no common minimizer for these functions, one shall rely in
another notion of optimality.
Up to now, in this settings, the most useful definition of optimality is
that of Pareto~\cite{paretobook}:
A point is a Pareto optimal if its image is minimal in the image
of the feasible set with respect to the
componentwise (partial) order.
Vector optimization in a generalization of MO where a closed convex cone is used to define
the partial order for which minimal elements are to be computed. If the cone is the positive
orthant one retrieves the componentwise order of MO.

\emph{Descent methods} for multiobjective and vector optimization
is, presently, an area of intense
research
( see~%
\cite{MR1778656,%
MR2049673,%
MR2178482,%
MR2515788,%
MR2835535,%
MR2860289,%
MR2886144,%
MR3029305,%
MR3175444,%
MR3093459,%
MR3166148,%
Fukuda2014585,%
MR3391220,%
%
MR3556811}
and the references therein).
These are iterative methods
in which 
\emph{all} objective functions  decrease
along the generated sequences.
As far as we now know,
the first one of these methods was proposed by Mukai in~\cite{MR0567375}.
In this work,
six descent methods for
\emph{constrained} MO were proposed,
three of which become, in the unconstrained case,
the steepest descent method.
In his pioneer work, Mukai neither studied the continuity properties of
the MO steepest descent direction nor considered the convex case,
proving that any limit point of the generated sequences satisfies a first
order necessary condition for optimality.

Fliege and Svaiter, independently, reinvented the steepest descent
direction and method~\cite{MR1778656} for unconstrained MO and proved
continuity of the this direction  for smooth objective
functions.
They also established 
convergence of the generated sequences in the convex case,
for 
stepsizes computed using an Armijo-type linesearch, under suitable
assumptions (e.g.\ boundedness of level sets).
In this manuscript we are concerned with the continuity properties of
the steepest descent direction for MO.
We will show that when the objective functions have Lipschitz continuous
gradients this direction
 is locally H\"older continuous with optimal exponent
$1/2$. Hence, in this case, the MO steepest descend direction 
is not Lipschitz continuous.

\section{Basic Definitions and Results}
\label{sec:bdr}

From now on $\Omega$ is a subset of $\R^n$, $f_1,\dots,f_m$
are scalar function on $\Omega$, and
\begin{align}
  \label{eq:f}
  f:\Omega\to\R^m, \quad f(x)=(f_1(x),\dots, f_m(x)).
\end{align}
The MO  problem is
\begin{align}
  \min \;\; f(x) \qquad x\in C
\end{align}
where $C\subseteq \Omega$ is the feasible set.
A feasible point $\bar x$ is a Pareto optimum for this problem if
\begin{align*}
  x\in C,\;f(x)\leq f(\bar x)\Rightarrow f(x)=f(\bar x).
\end{align*}
Henceforth we assume that
\begin{description}
\item[A1] $C=\Omega$ is open;
\item[A2] each $f_i$, $i=1,\dots,m$, is differentiable
\end{description}

A vector $v\in \R^n$ is a \emph{multiobjective descent direction}
at $x$
if
\[
\nabla \inner{f_i(x)}{v}<0,\quad i=1,\dots,m.
\]
It is trivial to verify that if $v$ is a descent direction at $x$,
then
\[
f_i(x+tv)<f_i(x)\qquad i=1,\dots, m
\]
for $t>0$ small enough, whence, at Pareto optimal points in the interior of the
feasible set there is no descent direction.
Following~\cite{MR1778656}, a point $x$ where
there is no descent direction will be called Pareto critical.
Since Pareto criticality is a necessary condition for optimality
(in the unconstrained case), is worth observing that
\begin{align*}
  x\text{ Pareto criticall}\iff
  \mathrm{range}(Df(x))\cap(-\R_{+ +}^m)=\emptyset.
\end{align*}
We believe that the condition at the right hand-side of the above
implication seems to be an acceptable extension of the notion of criticality for
scalar functions.
A natural question is how to compute or choose a ``reasonable''
descent direction at a given point.
\begin{definition}[Mukai~\cite{MR0567375}, Fliege \& Svaiter~\cite{MR1778656}]
  \label{df:sdd}
  The steepest descent direction for $f$ at $x\in\Omega$ is
  $\Lambda f(x)$,
  \begin{align}
    \label{eq:vfx}
   \Lambda f(x):= \argmin_{v\in \R^n}\;\max_{i=1,\dots,m}\inner{v}{\nabla f_i(x)}+\dfrac12\norm{v}^2\;.
  \end{align}
\end{definition}

It is convenient to define as in~\cite{MR1778656} the function
$\theta_f(x)$ as the optimal value of the functional whose minimizer
is $\Lambda f(x)$ in the above definition, that is,
\begin{align}
  \label{eq:alpha-or-theta}
  \theta_f(x)=\min_{v\in E}
  \;\max_{i=1,\dots,m}\inner{v}{\nabla f_i(x)}+\dfrac12\norm{v}^2\;.
\end{align}
Both $\Lambda_f(x)$ and $\theta_f(x)$ can be obtained minimizing a
convex quadratic function under linear constraints or solving its dual
by maximizing a concave quadratic function in in the unit simplex.

\begin{proposition}[\cite{MR0567375,MR1778656}]
  \label{pr:dual}
  For any,  $x\in\Omega$
  $(\tau,v)=(\theta_f(x),\Lambda f(x))$ is the solution of
  \begin{align}
    \label{eq:eqvfc}
    \begin{array}{ll}
      \min& \tau+\dfrac12 \norm{v}^2\\
      \text{s.t.}&\;\inner{v}{\nabla f_i(x)}\leq \tau,
  \end{array}
  \end{align}
  a problem whose dual is
  \begin{align}
    \label{eq:dual}
    \begin{array}{ll}
      \max
      &-(1/2)
         \Norm{\sum_{i=1}^m\alpha_i\nabla f_i(x)}^2\\[.6em]
      \text{s.t.}
      & \displaystyle\sum_{i=1}^m\alpha_i=1,\;\;\alpha\geq 0.
    \end{array}
  \end{align}
\end{proposition}

\begin{proof}
  Equivalence between \eqref{eq:vfx} and \eqref{eq:eqvfc} as well as
  optimality of $(\theta_f(x),\Lambda_f(x))$ for the second problem
  holds trivially.
  The 
  Lagrangian of \eqref{eq:eqvfc} is
  \begin{align}
    \label{eq:lg}
    \mathcal{L}((\tau,v),\alpha)=
    \tau+\norm{v}^2+\sum_{i=1}^m\alpha_i(\inner{\nabla f_i(x)}{v}-\tau),
  \end{align}
  which trivially implies the second part of the proposition.
\end{proof}

Proposition~\ref{pr:dual} has two \emph{quite trivial}, albeit interesting,
consequences.

\begin{corollary}
  \label{cr:mnorm}
  For any $x\in\Omega$, $\theta_f(x)=-(1/2)\norm{\Lambda_f(x)}^2$
  and $-\Lambda_f(x)$ is the minimal norm element in the convex
  hull of $\set{\nabla f_1(x),\dots,\nabla f_m(x)}$, that is, of the set
   \begin{align*}
     \overline{U}=\Set{u\;:\; u=\sum_{i=1}^m\alpha_i\nabla f_i(x)\,;\;\;  
     \alpha_i\geq 0,\,i=1,\dots,m;\;\;\sum_{i=1}^m\alpha_i=1 }
   \end{align*} 
   (hence $-\Lambda_f(x)$ is the orthogonal projection of the origin onto
   $\overline{U}$).
\end{corollary}

Next we revise some useful properties of the multiobjective steepest
descent direction.

\begin{lemma}[\mbox{\cite[Lemma 1]{MR1778656}}
  ]
  For any  $x\in \Omega$,
  \begin{enumerate}
  \item if $x$ is Pareto critical, then $\theta_f(x)=0$ and
    $\Lambda f(x)=0$;
  \item If $x$ is not Pareto critical then $\theta_f(x)< 0$,
    $\Lambda f(x)\neq 0$ and
    \begin{align*}
      \inner{\Lambda f(x)}{\nabla f_i(x)}\leq -\dfrac12\norm{\Lambda f(x)}^2
      \qquad i=1,\dots,m.
    \end{align*}
  \end{enumerate}
  If, additionally $\nabla f_i$, $i=1,\dots,m$, are continuous, then
  $x\mapsto \Lambda f(x)$ and $x\mapsto\theta_f(x)$ are continuous.
\end{lemma}

\section{H\"older continuity of the MO steepest descent direction}

The main result of this work is that, 
 under the
assumption of Lipschitz continuity of the objective functions'
gradients, the MO steepest descent direction is
locally H\"older continuous with optimal exponent $1/2$.
We will show that in general, even for gradients with
polynomial components, the MO steepest descent direction and,
equivalently, the minimal norm element on the convex hull of the
gradients of the objective function, fails to be Lipschitz continuous.

\begin{theorem}
  \label{th:holder}
  Suppose that $W\subseteq \Omega$ is convex, bounded, and
  $\nabla f_i$, $i=1,\dots,m$, are $L$-Lipschitz continuous on $W$,
  that is,
  \begin{align}\label{eq:hlip}
    \norm{\nabla f_i(y)-\nabla f_i(z)}\leq L\norm{y-z}
    \qquad \forall y,z\in W,\;i=1,\dots,m.
  \end{align}
  Then
  \begin{enumerate}
  \item $x\mapsto\Lambda f(x)$ is H\"older continuous on $W$;
  \item $x\mapsto\norm{\Lambda_f(x)}$ is Lipschitz continuous on $W$.
  \end{enumerate}
\end{theorem}

\begin{proof}
  Define, as in \cite[Section 3]{MR1778656}, for $x\in\Omega$
  \begin{align}
    \label{eq:gx}
    \phi_x(v)=\max\set{\inner{\nabla f_i(x)}{v}\;:\;i=1,\dots,m}\qquad
    (v\in\R^n).
  \end{align}
  Observe that $\phi_x$ is a convex sublinear functional and
  for $y,z\in W$, $\abs{\phi_y(v)-\phi_z(v)}\leq L\norm{y-z}\norm{v}$.
  Let
  \begin{align*}
    M=\max_{i=1,\dots,m,\;x\in W}\norm{\nabla f_i(x)}.
  \end{align*}
  Since $W$ is bounded,
  $M<\infty$.

  Take $y,z\in W$ and let $v_y=\Lambda f(y)$, $\alpha_y=\theta_f(y)$,
  $v_z=\Lambda f(z)$ and $\alpha_z=\theta_f(z)$.
  In view of Definition~\ref{df:sdd} and \eqref{eq:gx},
  the solution and optimal
  value of the problems
  \begin{align}
    \min \;\phi_y(v)+\dfrac12\norm{v}^2,\qquad
    \min \;\phi_z(v)+\dfrac12\norm{v}^2, \quad\qquad v\in\R^n
  \end{align}
  are, respectively, $v_y,\alpha_y$ and $v_z,\alpha_z$. Therefore,
  \begin{align*}
    \phi_y(v)+\dfrac12\norm{v}^2
    & \geq
    \phi_z(v)+\dfrac12\norm{v}^2-L\norm{y-z}\norm{v}
    \\
    & \geq
    \alpha_z+\dfrac12\norm{v-v_z}^2 -L\norm{y-z}\norm{v}
  \end{align*}
  where the first inequality follows~\eqref{eq:hlip} and \eqref{eq:gx}
  and the second inequality follows from 
  the $1$-strong convexity of $v\mapsto g_z(v)+\norm{v}^2/2$
   and the optimality of $v_z$ for this function.
  Substituting $v_y$ for $v$ in the above inequalities we conclude that
  \begin{align*}
    \alpha_y\geq\alpha_z+\dfrac12\norm{v_y-v_z}^2-L\norm{y-z}\norm{v_y}.
  \end{align*}
  By the same token,
  \begin{align*}
    \alpha_z\geq\alpha_y+\dfrac12\norm{v_y-v_z}^2-L\norm{y-z}\norm{v_z},
  \end{align*}
  Adding the above inequalities we conclude that
  \begin{align*}
    \norm{v_y-v_z}^2\leq L\norm{y-z}(\norm{v_y}+\norm{v_z}).
  \end{align*}
  Since $W$ is bounded,
  $M=\max_{i=1,\dots,m,\;x\in W}\norm{\nabla f_i(x)}<\infty$.  It
  follows from Corollary~\ref{cr:mnorm} that $\norm{v_y}\leq M$ and
  $\norm{v_z}\leq M$.  Hence
  \begin{align*}
    \norm{\Lambda_f(y)-\Lambda_f(z)}=
    \norm{v_y-v_z}\leq \sqrt{2LM}\;\norm{y-z}^{1/2},
  \end{align*}
  which proves the H\"older
  continuity of $x\mapsto \Lambda f(x)$ on $W$ with exponent $1/2$.

  To prove item 2, let $\overline{U}_y$ and $\overline{U}_z$ be the
  convex hulls of the gradients of the objective functions at $y$ and $z$,
  respectively.
  The minimal norm element of $\overline{U}_y$ is
  \begin{align*}
    -\Lambda_f(y)=\sum_{i=1}^m\alpha^*_i\nabla f_i(y)
  \end{align*}
   for some $\alpha^*\in\R^m_+$ such that $\sum_{i=1}^m\alpha^*_i=1$
   Define $\tilde u=\sum_{i=1}^m\alpha_i^*\nabla f_i(z)$.
   Then $ \tilde u\in\overline{U}_z$ and
  \begin{align*}
    \norm{-\Lambda_f(y)-\tilde u} \leq
    \sum_{i=1}^m\alpha_i^*\norm{\nabla f_i(y)-\nabla f_i(z)}
    \leq
    \sum_{i=1}^m\alpha_i^*L\norm{y-z}=L\norm{y-z}
  \end{align*}
  Since $-\Lambda_f(z)$ is the minimal norm element of $\overline{U}_z$,
  \begin{align*}
    \norm{\Lambda_f(z)}\leq\norm{\tilde u}\leq\norm{\Lambda_f(y)}
    + \norm{-\Lambda_f(y)-\tilde u}\leq\norm{\Lambda_f(y)}
    +L\norm{y-z}
  \end{align*}
  By the same token,
  $\norm{\Lambda_f(y)}\leq\norm{\Lambda_f(z)}+L\norm{y-z}$ and the
  conclusion follows.
\end{proof}

Finally, we establish the optimality of the H\"older exponent $1/2$.

\begin{proposition}
  \label{lm:not-lip}
  Under the assumptions of Theorem~\ref{th:holder},
  the H\"older exponent $1/2$ derived there
  can not be improved.
\end{proposition}

\begin{proof}
  Let
  \begin{align}
    \label{eq:cex}
    \Omega=\R^2,\;\; f_1(r,s)=\dfrac{r^2+s^2}2,\;\;
    f_2(r,s)=r, \;\;\;\;f(x)=(f_1(x),f_2(x))\;\;
    \text { for }x\in\R^2,
  \end{align}
  and 
  define, for $0<t<\pi/2$,
  \begin{align}
    \label{eq:xpxpp}
    y_t=\cos t \;  (\cos t,\sin t),\qquad
    z_t=(1,\cos t\,\sin t).
  \end{align}
  Direct use of these definitions yields
  \begin{align*}
    \nabla f_1(y_t)-\nabla f_2(y_t)=y_t-(1,0)=(-\sin^2t,\sin t \cos t)
    =\sin t(-\sin t,\cos t)\perp \nabla f_1(y_t)
  \end{align*}
  and
  \begin{align*}
    \nabla f_1(z_t)-\nabla f_2(z_t)=z_t-(1,0)=(0,\cos t\sin t)\perp
    \nabla f_2(z_t).
  \end{align*}
  Therefore, $\nabla f_1(y_y)$ is the minimal norm element on the segment
  $[\nabla f_1(y_t),\nabla f_2(y_t)]$,
  $\nabla f_2(z_t)$   is the minimal norm element on the segment
  $[\nabla f_1(z_t),\nabla f_2(z_t)]$, and it follows from
  Corollary~\ref{cr:mnorm} that
  \begin{align*}
    \Lambda_f(y_t)=-\nabla f_1(y_t)=-y_t,
    \qquad
    \Lambda_f(z_t)=-\nabla f_2(z_t)=-(1,0).
  \end{align*}
  Combining the above equalities with \eqref{eq:xpxpp}
  and the assumption $0<t<\pi/2$ we conclude that
  \begin{align*}
    \norm{\Lambda_f(y_t)-\Lambda_f(z_t)}=\norm{(1,0)-y_t}=\sin t,\qquad
    \norm{y_t-z_t}=(\sin t)^2.
  \end{align*}

  Let $V$ be a neighborhood of $(1,0)$ and $\eta\in(0,1]$.
  Since $y_t\in V$ and $z_t\in V$ for
  $t>0$ small enough
  \begin{align*}
    \sup\Set{
    \dfrac{\norm{\Lambda f(y)-\Lambda f(y)}}{\norm{y-z}^\eta}\;\colon
    y,z\in V, \;y\neq z
    }
    &\geq\lim\sup_{t\to 0^+}
      \dfrac{\norm{\Lambda f(y_t)-\Lambda f(z_t)}}{\norm{y_t-z_t}^\eta}\\
    &=\lim\sup_{t\to 0^+}\dfrac{\sin\,t}{(\sin\,t)^{2\eta}}.
  \end{align*}
   To end the proof, observe that the above $\lim\sup$ is $+\infty$
   for $\eta\in (1/2,1]$.
\end{proof}



\end{document}